\documentclass[a4paper,11pt]{amsart}

\usepackage{amssymb,amsbsy,amsmath,amsfonts,amssymb,amscd}
\usepackage{latexsym}
\usepackage{mathtools}
\usepackage{graphics}
\usepackage{color}
\usepackage{comment}
\input xy
\xyoption{all}

\theoremstyle{plain}
\newtheorem{thm}{Theorem}[section]
\newtheorem{theorem}[thm]{Theorem}

\newtheorem{lemma}[thm]{Lemma}

\theoremstyle{definition}

\newtheorem{example}[thm]{Example}

\numberwithin{equation}{section}

\newcommand{\sB}{{\mathcal B}}

\newcommand{\sE}{{\mathcal E}}

\newcommand{\sH}{{\mathcal H}}

\newcommand{\sJ}{{\mathcal J}}
\newcommand{\sK}{{\mathcal K}}
\newcommand{\sL}{{\mathcal L}}

\newcommand{\PP}{\ensuremath{\mathbb{P}}}

\newcommand{\CC}{\ensuremath{\mathbb{C}}}

\newcommand{\ZZ}{\ensuremath{\mathbb{Z}}}

\newcommand{\hol}{\ensuremath{\mathcal{O}}}
\newenvironment{dedication}
        {\begin{quotation}\begin{center}\begin{em}}
        {\par\end{em}\end{center}\end{quotation}}

\newcommand\la{\lambda}
\newcommand\s{\sigma}

\newcommand\Ga{\Gamma}
\newcommand\De{\Delta}

\newcommand\de{\delta}
\newcommand\e{\epsilon}

\newcommand{\Lam}{\Lambda}

\DeclareMathOperator{\Pic}{Pic}

\DeclareMathOperator{\Hom}{Hom}

\DeclareMathOperator{\Heis}{Heis}

\newcommand{\ra}{\ensuremath{\rightarrow}}

\def\eea{\end{eqnarray*}}
\def\bea{\begin{eqnarray*}}

\newcommand\dual{\mathrel{\raise3pt\hbox{$\underline{\mathrm{\thinspace d
\thinspace}}$}}}
\newcommand\qe{\ifhmode\unskip\nobreak\fi\quad $\Box$}       

\def\BOX{\hfill\lower.5\baselineskip\hbox{$\Box$}}

\newtheorem{theo}{Theorem}[section]

\newtheorem{remarkk}[theo]{Remark}
\newenvironment{rem}{\begin{remarkk}\rm}{\end{remarkk}}

\newtheorem{prop}[theo] {Proposition}
\newtheorem{cor}[theo]{Corollary}

\usepackage{hyperref}
\title [
$(1,3)$-Polarizations ]{ The Hesse pencil and polarizations of type $(1,3)$ on Abelian surfaces.}

\author{ Fabrizio Catanese, Edoardo Sernesi}
\address {Mathematisches Institut der Universit\"at Bayreuth\\
NW II,  Universit\"atsstr. 30\\
95447 Bayreuth}
\email{fabrizio.catanese@uni-bayreuth.de}
\address{  Korea Institute for Advanced Study, Hoegiro 87, Seoul, 
133--722.}
\address{  Universit\'a di Roma Tre, 
L.go S.L. Murialdo 1, 00146 Roma (Italy).}
\email{sernesi@gmail.com}
 
 \date{\today}

\begin{document}

\maketitle

\begin{dedication}
This article is dedicated to the memory of Alberto Collino.
\end{dedication}

\begin{abstract}
In this short note we prove two theorems, the first one is a sharpening of a result of Lange and Sernesi \cite{la-se}:
the discriminant curve W of a general Abelian surface $A$ endowed with an irreducible  polarization $D$
of type $(1,3)$
is an irreducible curve of degree $18$ whose singularities are exactly  $36$ nodes and $72$ cusps.
Moreover, we analyze the degeneration of the discriminant curve $W$ and its singularities as $A$ tends to the product of two  equal elliptic curves.

 The second theorem, using the first one in order to prove a transversality assertion, shows that the general element of a family of surfaces constructed by Alessandro and Catanese \cite{al-cat}
 is a smooth surface, thereby proving the  existence of a new family of minimal surfaces 
 of general type with $p_g=q=2, K^2=6$ and Albanese map of degree $3$.
\end{abstract}


\section*{Introduction.}

The curious reader will certainly immediately  ask: what do the Hesse pencil of plane cubic curves and 
Abelian surfaces $A$ with an ample divisor $D$ of type $(1,3)$ have in common?

The answer is easy: the Heisenberg group $\sH_3$.

In fact, both in case of an elliptic curve embedded in $\PP^2$ as a cubic $E_{\la}$
in the Hesse pencil, or of an Abelian surface mapped to   $\PP^2$ via the morphism $\varphi_D$
associated to $V: = H^0(\hol_A(D))$,
there is an action of the Heisenberg group on $V$ and   $\PP^2 = \PP(V)$.

And indeed an action of the extended Heisenberg group $\hat{\sH}$ which induces an action on $\PP^2$
of the Quotient $\hat{G}$ of  $\hat{\sH}$ by its centre:  $\hat{G}$ in this case is a semidirect product of
$\mu_3$ (generated by $(y_1, y_2, y_3) \mapsto (y_1, \e y_2, \e^2 y_3)$, $\e^3=1$)  with 
the symmetric group $\mathfrak S_3$.

In both cases $\hat{\sH}$ induces an action on the elliptic curve, respectively the Abelian surface $A$,
by the group $\hat{G}$,  generated by multiplication by $-1$ and by a group of translations isomorphic to 
$ G =  \mu_3 \times (\ZZ/3)$.

While the geometry of the Hesse pencil is classically well known, in the case of our Abelian surfaces
$A$ the geometry of the discriminant
curve $ W \subset \PP(V)$ (consisting of the singular curves in the linear system $|D|$),
and its dual curve, the branch curve $\sB$ of $\varphi_D$, was only recently addressed by Birkenhake and Lange \cite{b-l-13}, 
Casnati \cite{casnati},  Lange and Sernesi \cite{la-se}.

Birkenhake and Lange  \cite{b-l-13} calculated both curves $W, \sB$  in a special case, where $A $ is a product $E \times E$
of two equal elliptic curves, and $D$ is a divisor given by the sum of three curves, namely
 $ E \times \{0\},  \{0\} \times E$ and the antidiagonal.  
 
 They showed that in this case $W =: W_0 $ is the sum of the cubic $E$ counted with multiplicity $3$
 and its $9$ flex tangents, while $\sB_0$ is the dual sextic $E^{\vee}$ counted with multiplicity $3$.
 
 We first show here that if $E$ is not isomorphic to the Fermat (equianharmonic) elliptic curve, 
 then $W_0$ has as singular set $E$ and 36 nodes.
 
 For a general surface $A$, we get again two curves $W_t, \sB_t$ of degree $18$, and our main motivation was
 to see how does for instance $W_t$ and its singularities degenerate to $W_0$ in the limit process.
 
 The answer is the  following small improvement of the main  result of \cite{la-se}:

\begin{thm}
For a general pair $(A_t,D_t)$ of an Abelian surface with an ample divisor $D$ of type $(1,3)$ the
discriminant curve $W_t$ is  an irreducible curve of degree 18, whose singularities are
36 nodes and 72 cusps. As $A_t$ tends to the product of two equal elliptic curves $A_0 = E\times E$, with $E$ not equianharmonic,
in which case $W_0$ consists of the cubic E counted with multiplicity 3 and the nine flex-tangents
$L_1, \dots, L_9$, the 36 nodes tend to the 36 nodes $L_i \cap L_j$ while the cusps tend, in groups
of eight, to the 9 flex-points $P_1, \dots, P_9$.
\end{thm}

\bigskip

We give now an application of the previous Theorem.

\bigskip

In the article \cite{al-cat} the following family of surfaces was constructed: the family of quotients $ S = S'/G$
of the surfaces $S'$ in the family
 $$ S'  \subset \PP^2 \times A,  S' : = \{ (y,z) \in \PP( V) \times A | \sum_j y_j x_j(z) = 0 , \sum_i y_i^2 y_{i+1} =0\}, $$
where $ y: = ( y_1, y_2,y_3) \in \PP^2,$ $\{x_1, x_2,x_3\} $ is a canonical  basis of $V = H^0(A, \hol_{A}(D))$,
and $D$ is a divisor of type $(1,3)$. 

It was  observed that  $S' \subset C \times A$, where $C$ is the cubic curve 
$$ C: = \{ y \in \PP^2 | \sum_i y_i^2 y_{i+1} =0\},$$ 
and that $S$ has irregularity $ q=2$ since $G$ does not act by translations on $C$.

The quotients $S$, called 
AC3 surfaces, have  then $ p_g=q=2, K^2=6,$ and degree $d=3$ of the Albanese map.

In \cite{al-cat} was missing a proof that,  for general $(A,D)$, $S'$ is smooth. 
For this it is sufficient (ibidem) to show that $C$ intersects transversally the discriminant 
curve $W_t$ of a general pair $(A_t, D_t)$, and equivalently, that $C^{\vee}$ intersects transversally
the branch curve $\sB_t$.

We establish this property 
in  the present article, and prove:

\bigskip

\begin{thm}
		The general AC3 surface   $S$ is smooth, hence this family
		yields a    new irreducible component 
		of the moduli space of surfaces of general type with $p_g=q=2$, $K^2=6$,and with Albanese map 		of degree $d=3$.

	\end{thm}

\section{Notation}

1) Given a finite Abelian group $H$, the finite  Heisenberg group $ \Heis(H)$ is the central extension 
$$ 1 \ra \mu_n \ra \Heis(H) \ra H \times H^* \ra 1,$$
where $\mu_n \subset \CC^* $ is the group of $n$th roots of $1$, $n$ is the exponent of $H$,
$ H^* : = \Hom (H , \CC^*)$, and $\Heis(H)$ is the group generated by the respective actions
of $h \in H$ on $ \CC^H$ given by translation, $(h\cdot f) (x) = f (x + h)$, and of $ \chi \in H^*$
given by multiplication with the character $ (\chi\cdot f)(x) = f(x) \chi(x)$.

This representation of $\Heis(H)$ on $ \CC^H$ is called the {\bf Schr\"odinger representation}.

 If  $H$ is a cyclic group $H = \ZZ/n$, we  denote $\Heis(H)  = : \sH_{n}$.

\bigskip

2) If $A$ is  an Abelian surface endowed with an ample  divisor class $D$ yielding a polarization with elementary divisors $\de_1,\de_2$ (recall that $\de_1 | \de_2$ and $h^0 (\hol_A(D)) = \de : = \de_1 \de_2$),
one has a surjective morphism (\cite{mumford}) 
\[
\begin{split}
\Phi_D : &A \ra A^{\vee} : = \Pic^0(A) \\
&x \longmapsto t_x^* (D) - D,
\end{split}
\]
 and its kernel $G: =  \sK(D) \cong H^2$, where $ H := (\ZZ/ \de_1)\oplus (\ZZ/ \de_2)$ and the Heisenberg
group $\sH_D : = \Heis(H)$ is a group of isomorphisms of the line bundle $\sL:=\hol_{A}(D)$ mapping onto $G$.

 $V: = H^0(A, \hol_{A}(D))$ is isomorphic to the Schr\"odinger representation of $\sH_D$.

\bigskip

 3) Given an ample divisor $D$ on an Abelian surface $A$, the linear system $|D|$ has no base points
 by the theorem of Lefschetz if $\de_1 \geq 2$. 
 
 If $\de_1=1$ and  $\de \geq 3$, then $|D|$ 
 has no base points 
 if it has no fixed part; since the base-point locus $\Sigma$ is $G$-invariant, hence it has cardinality a multiple
 of $|G| = \de^2$, while $D^2 = 2 \de$. 
 
 And  the system $D$ has no fixed part (see \cite[Lemma 10.1.1]{b-l})  unless the pair $(A,\hol_{A} (D))$
 is isomorphic to a polarized product  of two elliptic curves, 
 $$(A,\hol_{A} (D))  \cong (E_1 ,\hol_{E_1} (D_1 )) \times (E_2 ,\hol_{E_2} (D_2 )),$$
 where $ deg (D_1)=1$, $deg(D_2)= \de_2$.

\section{Geometry of the action  of the Heisenberg group $\sH_3$ on the plane,
and the Hesse pencil.}\label{Sheishess}

Consider the  action on $\PP^2$ with coordinates $(y_1, y_2, y_3)$ of the homomorphic image 
$\hat{G}$  of the extended Heisenberg group $\hat{\sH_3}$, where
 $\hat{G}$ is the quotient of  $\hat{\sH}$ by its centre. 
 
 $$ \hat{G} = \mu_3 \rtimes \mathfrak S_3,$$
 here 
$\mu_3$ acts  by $(y_1, y_2, y_3) \mapsto (y_1, \e y_2, \e^2 y_3)$, where $\e^3=1$,
and 
the symmetric group $\mathfrak S_3$ acts by permuting the coordinates.

The subgroup $G = \mu_3 \times \ZZ/3$ acts on $\PP^2$, where   $\ZZ/3$ acts by a cyclical permutation of the
cordinates. We let $\s$ be the permutation $ (1,2)$, hence $\s$ permutes $y_1$ with $y_2$,
and $ \hat{G} = G \cup \s G$; and we let $\tau$ be a generator of $\mu_3$.

Then each monomial of degree $3$ is an eigenvector for the action of $\tau$, and the eigenspaces are
respectively spanned by $\{y_i^3, y_1 y_2 y_3\}$, $\{y_i^2 y_{i+1}\}$, $\{y_i^2 y_{i-1}\}$ (here the indices  have to be taken in $\ZZ/3$).
Hence 

\begin{enumerate}
\item
Every $\hat{G}$-invariant cubic is an element of the Hesse pencil (here $\la := \la_1 / \la_0$)
$$E_{\la}: = \{ y \in \PP^2, y : = (y_1, y_2, y_3) | \la_0 \sum_j y_j^3 + 6 \la_1 y_1 y_2 y_3=0\}.$$
\item
The only other $G$-invariant cubics are the cubics 
$$ C : = \{ y \in \PP^2 | \sum_j y_j^2 y_{j+1} = 0\}, \ C'  : = \{ y \in \PP^2 | \sum_j y_j^2 y_{j-1} = 0\},$$
and these  are exchanged by $\s$.
\item
The $9$ base points $P_1, \dots, P_9$ of the pencil are the $G$-orbit of $P_1 : = (1,-1,0)$ (which is stabilized by $\s$)
and are the flex points of each $E_{\la}$.
\item
The group $\hat{G}$ acts on each smooth cubic $E_{\la}$ via the action which, using the group law on an elliptic curve,
can be written as $ z \mapsto \pm z + \eta$, 
where $\eta$ is any 3-torsion point in $Pic^0(E_{\la})$.
\item
There are exactly four  singular cubics in the Hesse pencil, and these are four  triangles:
since the Hessian curve of $E_{\la}$ is the curve $E_{\nu}$, $\nu : = - \frac{1 + 2 \la^3}{  6 \la^2}$,
these are the curves with $\la = \nu$, that is, $ (2 \la)^3 = -1$ or $\la = \infty. $
\item
On each triangle the vertices and the sides are stabilized by a subgroup of order $6$, while the
points on a side but different from the flexpoints have trivial stabilizer.
\item
The vertices of the four triangles yield  the four respective $G$-orbits of the points:
$$ (1,0,0), \ (1,1,1) , \ (1,1,\e), \ (1,1,\e^2).$$
\item 
The other points (that is, different from flexpoints and vertices) with a non trivial stabilizer
are the respective fixpoints of the order $2$ transformations   $ g \s $;  since
$\s g \s = g^{-1}, g \s  = g^{-1} \s g $ and these fixpoint sets are the orbit of  
 $ Fix(\s) = \{(1, -1, 0)\} \cup \{ y_1-y_2=0\}$, hence
 \item
 the points with nontrivial stabilizers are the nine flexpoints, the 12 vertices, and
 the points of the nine lines $y_{j+1} = \e^i y_j$.
 \item
 These nine lines intersect the smooth cubics of the pencil in 27 points, which, together with the nine flexpoints,
  are on the elliptic curve the solutions of $ 2 z = \eta$, for $\eta $ a 3-torsion point (hence are sums of a 2-torsion with a 3-torsion point);
  and these nine lines  intersect the triangles in the 3 vertices. 
 \item
 One concludes that all orbits have cardinality $18$, with  the exceptions of the points 
 on the nine lines $y_{j+1} = \e^i y_j$, which generally have orbit of cardinality $9$,
 of four orbits of cardinality 3 (vertices of one of the four  triangles in the Hesse pencil), and the 9
 flexpoints; moreover these nine lines intersect in the 12 vertices.

\end{enumerate}

\subsection{Intersections of the nine flex tangents of a smooth plane cubic}\label{SSintersections}
Here, we take a smooth cubic $E_{\la}$ and we take the nine lines $L_1, \dots, L_9$
which are tangent to $E_{\la}$  at the flexpoints $P_1, \dots, P_9$.

We expect that we shall obtain $36$ intersection points, but there is one exception.

\begin{prop}
The  set $\{ L_i \cap L_j| i \neq j\}$ consists of exactly 36 points if and only if $E_{\la}$ 
is not isomorphic to the Fermat elliptic curve $E_0$.
\end{prop}
\begin{proof}
Assume that $E$ is a smooth cubic, and that a point $P$ lies on three distinct  flextangents
$L_1, L_2, L_3$. Then the polar conic $Q_P$ passes through the corresponding flexpoints $P_1, P_2, P_3$.
A local calculation shows that the intersection number of $E$ and $Q_P$ at $P_i$ equals $2$,
hence $Q_P$ cannot be smooth: otherwise there would be 3 tangents (the lines $L_j$)
passing through $P$, a contradiction. Hence $Q_P$ consists of two lines. These cannot be distinct, otherwise
each $P_i$ is either a singular point of $Q_P$ or $L_i \subset Q_P$:  this can only happen   if
$Q_P = L_i + L_j$, and, for $k \neq i,j$, $ P_k \in L_i \cap L_j$. However, this is manifestly impossible,
since $L_i \cap E = \{ P_i\}$. The conclusion is that $Q_P$ is a double line, $Q_P = 2 L$.

There are two possibilities: $ P \in L$, or $P \notin L$.

We can take coordinates such that $P= (0,0,1)$ and $L$ is $x=0$, respectively $z=0$.

Letting $f=0$ be the equation of $E$, we know that $Q_P$ has equation $\frac{\partial f }{\partial z }=0$. 

In the first case $\frac{\partial f }{\partial z }=x^2$, hence $ f = x^2 z + \phi (x,y)$, hence $ P \in E$,
a contradiction.

In the second case $\frac{\partial f }{\partial z }=3 z^2$, hence $ f = z^3 +  \phi (x,y)$. But $\phi(x,y)$
has three distinct roots (corresponding to the three points $P_1, P_2, P_3$). 
Changing coordinates we may assume that  $\phi (x,y) = x^3 + y^3$.

Then $E$ is the Fermat cubic, in suitable coordinates.

Conversely, if $E$ is the Fermat cubic,  by the previous calculations we find such a point $P$
whose polar is twice a line $L$ passing through 3 flexpoints.

By $\hat{G}$-symmetry,
it follows that the whole orbit of $P$ is made of points with such a property. Clearly we have
therefore, for $E_0$,  the three vertices of the triangle $\{ y_1y_2y_3=0\}$ whose polar is twice the opposite side of the triangle. 

\end{proof}

\begin{rem}

In general, given a point $P = (x_1, x_2, x_3)$, the family of polar conics 
of the cubics $E_{\la}$, $Q_P(\la) = \la_0 Q_P(0) + \la_1 Q_P(\infty)$
 is a linear system of conics
 $$  \sum_j x_j ( \la_0 y_j^2 + 2 \la_1  y_{j-1}y_{j+1})=0,$$
 whose determinant is again a cubic $\{x_1 x_2 x_3 (1 + 2 \la^3) - \la^2 \sum_j x_j ^3=0\}$,
 namely  the cubic of the Hesse pencil $E_{\nu (\la)}$ previously considered, and we have a rank one conic only for the singular points
 of   $E_{\nu (\la)}$.
 
 Hence, the only points $P$ which belong to three  flex tangents to some $E_{\la}$ 
 are the 12 vertices of the four triangles in the Hesse pencil; and this occurs exactly
 for the cubics which are corresponding to the four triangles, under the map
 for the Hesse pencil $ \la \mapsto \nu (\la) = - \frac{1 + 2 \la^3}{6 \la^2}$.
 
 When we look for smooth conics, we obtain $\la=0$, or $\la = \pm \e^i$, since  the solutions of
 $$ 2 \la^3 - 3 \e^i \la^2 + 1 =0$$
 are $\pm \e^i$ and $ - (1/2) \e^i$.   
 
 For these values we obtain a Fermat cubic (an equianharmonic cubic).

\end{rem}

An interesting question is how do the points $P_{i,j}(\la) : = L_i(\la) \cap L_j(\la)$ vary.

in general, we have:

\begin{lemma}
Given two homological pencils of lines through two points $P_i, P_j$ (that is, there is a projectivity between the
two pencils of lines, so that one can write each  pencil as a pencil $\{L_i(\la)\}$ of lines),  then the intersection 
of homological lines $P_{i,j}(\la) : = L_i(\la) \cap L_j(\la)$ move in a conic, which contains the line $L$  joining $P_i, P_j$
if and only if this line is self homological.
\end{lemma}
\begin{proof}
If the line $L$ is self homological, there are coordinates such that the intersection point is determined  by the equations 
$(x,y) = (x,z) \in \PP^1$, hence either $x=0$ or $y-z=0$.

Otherwise, we have the condition $(x,y) = (z,x) \in \PP^1$, hence 
the intersection point satisfies $ x^2 - y z=0$.

\end{proof}

In our situation, given two flexpoints $P_i, P_j$,  the line $L$  joining them contains  a third flexpoint $P_k$,
hence there is a triangle in the Hesse pencil having $L$ as tangent at these points, whence we are in the first situation of the previous Lemma, and the points $P_{i,j}(\la)$,
for $E_{\la}$ different from this triangle, move in a line.

Because of $\hat{G}$ symmetry, we may assume that the first flexpoint $P_i$  is $P_1 = (1,-1,0)$,
and $P_j$ is either $(0,1,-1)$ or $(1, -\e, 0)$.

In the first case the intersection point satisfies $$  y_1 + y_2 - 2 \la y_3= 2 \la y_1 - y_2 - y_3 = 0,$$
whence $ y_1 -y_3 =0$. In the second case the intersection point satisfies
$$ y_1 + y_2 - 2 \la y_3=y_1 + \e^2 y_2  - 2 \e \la y_3= 0,$$
whence $y_1-y_2=0$.

It follows then easily:

\begin{cor}
For $\la \neq 0, \infty, \pm \e^i, - \frac{1}{2}  \e^i$, the intersection points of pairs of flex tangents to
the cubic $E_{\la}$  are 36 distinct points  $P_{i,j}(\la) : = L_i(\la) \cap L_j(\la)$. These belong to the 
nine lines  $y_{j+1} = \e^i y_j$ in groups of four. In particular, these 36 points form 4 $\hat{G}$-orbits of cardinality 9.

\end{cor}


\section{ The geometry of the discriminant curve and its degenerations}

Let $(A,\sL)  = (A,\hol_A(D))$ be a polarized abelian surface of type $(1,3)$ such that $|\sL|$ has no fixed components. The general element of $|\sL|$ is a nonsingular curve of  genus 4 and $\dim(|\sL|)=2$. 
After a choice of a  canonical  basis $s_1, s_2, s_3$ of $H^0(A,\sL)$ we will identify $|\sL|=\PP(H^0(A,\sL))$ with $\PP^2$, with homogeneous coordinates $(y_1,y_2,y_3)$  (hence to $y$ corresponds the section $\sum_i y_i s_i(z)$) and therefore we will identify   $\PP(H^0(A,\sL)^\vee)$ with $\PP^{2\vee}$ with dual coordinates $(x_1,x_2,x_3)$. 
The morphism \footnote{which we  also denote $\varphi_D$} 
$$
\varphi_\sL: A\ra \PP^{2\vee}
$$
is a $6:1$ cover, and it is given by 
$$\varphi_\sL (z) = (s_1(z), s_2(z), s_3(z))$$
 where $s_1, s_2, s_3$ is the chosen  canonical basis
(for the Heisenberg action). 

The geometry of this cover has been described in \cite{la-se}   for a  general $(A,\sL)$, by degenerating it to a special configuration $(A_0,\sL_0)$, whose geometry has been studied in detail in \cite{b-l-13}. In this section we outline what has been proved, referring to the above mentioned papers  for details.

\subsection{The incidence curve.}
This subsection is extracted from \cite{la-se}. Given any $(A,\sL)$ as above  we may consider the incidence curve $\Gamma\subset A\times \PP^2$, which is defined as follows. Let $$
\xymatrix{
A& A\times \PP^2 \ar[l]_-{\pi_1} \ar[r]^-{\pi_2}& \PP^2
}
$$
be the projections. Let $J_1(\sL)$ denote the first jet bundle of $\sL$: its fibre at a point $z \in A$ is the vector space $\sL\otimes \hol_A/\sJ_z^2$.
There is a natural homomorphism of sheaves 
$$
\sigma: \pi_2^*\hol_{\PP^2}(-1) \ra \pi_1^*J_1(\sL)
$$
 which associates to every local section the truncation at second order of its Taylor expansion. To $\sigma$ there corresponds a section of 
 $$
 \sE := \pi_2^*\hol_{\PP^2}(1)\otimes \pi_1^*J_1(\sL)
 $$
 The incidence curve $\Gamma$ is defined as the vanishing scheme of this section. Set theoretically
 $$
 \Gamma = \{(z,[C])\in A\times \PP^2|  z \ \ \text{is a singular point of}\ \  C = \{\sum_i y_i s_i (z)=0 \}\}.
 $$
 Every component of $\Gamma$ has at most dimension  one, because otherwise there would be infinitely many non-reduced  curves $[C]\in \PP^2$,  and this is absurd since we assume 
 that  $|D|$ is base point free.
 On the other hand, since $\sE$ is a vector bundle of rank three, every component of $\Gamma$ has codimension less than or equal to three. Therefore $\Gamma$  has pure dimension one, and it is a local complete intersection curve. Its main properties are:
 
 \begin{enumerate}
     \item $\Gamma$ has arithmetic genus $p_a(\Gamma)=28$. This follows from a Chern class computation. 

     \item $\pi_1(\Gamma) \subset A$ is the \emph{ramification divisor}  $R$ of $\varphi_\sL$. The inclusion $\pi_1(\Gamma)\subset R$ is obvious. On the other hand, if $p\in R$ then the pencil  $|\sL(-p)|$ consists of curves with a fixed tangent at $p$, or a singular  point at $p$, and therefore contains a curve singular there, hence $p\in \pi_1(\Gamma)$.

     \item The \emph{branch divisor} $\sB\subset \PP^{2\vee}$, the image of $R$ under $\varphi_\sL$,  is a plane curve of degree 18.  This follows since, by  Hurwitz' formula, $ R$ is linearly equivalent to $3D$ and because $ 3 D^2=18$. Hence, moreover, $p_a(R)=28$.
     \item $\pi_2(\Gamma)\subset \PP^2$ is the \emph{discriminant scheme of $|\sL|$}, denoted by $W$.  It is a plane curve, because $\Gamma$   is a curve, and parametrizes the singular curves in $|\sL|$.
      $W$ has degree $18$ by the Zeuthen Segre formula, because $D^2=6$ and the number of singular fibres in a general pencil
     in $|D|$ equals $ 6 + 4 (g-1) = 18$, since the curves of $|D|$ have genus $g=4$.

     \item $W$ contains the closure $\overline V_{\sL,1}$  of the \emph{Severi variety} $V_{\sL,1}\subset |\sL|$, which parametrizes the curves in $|\sL|$ having precisely one node and no other singularities. In general the inclusion $\overline V_{\sL,1}\subset W$ is strict (see the next subsection).

     \item $\overline V_{\sL,1}$ ha degree at most 18 and equality holds if and only if $W=\overline V_{\sL,1}$. In this case $\pi_2: \Gamma \ra W$ is birational on each irreducible component of $\Gamma$. In particular, if $\Gamma$ is reduced and $W=\overline V_{L,1}$ then $W$ is also reduced of degree 18.

     \item The group $\hat G$ acts on $A,\PP^2,\PP^{2\vee}$ equivariantly with respect to $\pi_1$ and $\pi_2$, so that   there are induced actions on $\Gamma,R,\sB,W$.
 \end{enumerate}

 Altogether we have the following basic configuration associated to any $(A,\sL)$ as above:
  \begin{equation}\label{E:config}
      \xymatrix{
 &&\Gamma \ar[dll]_-{\pi_1}\ar[dr]^-{\pi_2} \\
 R\ar[d]^-{\varphi_L}&&&W\\
 \sB
 }
  \end{equation}
The main result of \cite{la-se} is the following:

\begin{theorem}\label{T:mainls}
   If $(A,\sL)$ is a general abelian surface of type $(1,3)$ then 
   \begin{itemize}
       \item[(i)] $\Gamma$ is reduced and irreducible of arithmetic genus 28.
       \item[(ii)] both $\sB$ and $W$ are reduced and irreducible curves of degree 18 and geometric  genus $p_g(\sB)=p_g(W)=28$, with 72 cusps and 36 nodes, possibly infinitely near   (tacnodes).
       \item[(iii)] $\sB$ and $W$ are dual to each other.
       \item[(iv)]  $\pi_1:\Gamma\ra R$ and $\pi_2:\Gamma\ra W$ are birational.
   \end{itemize}
\end{theorem}
  
Later in this section we will remove from this statement the possibility that there are tacnodes. Note also that, since $p_a(\Gamma)= p_a(R) =p_g(W) = 28$ and $\pi_2$ is birational,  $\Gamma$  and $R$ must be nonsingular.

 \subsection{A special case of abelian surface of type $(1,3)$.}\label{SSspecial}
 This subsection describes results from \cite{b-l-13} and \cite{la-se}. Let $E$ be an elliptic curve. On the abelian surface $A_0:=E\times E$ let $\sL_0$ be the line bundle
 $$
 \sL_0 =\hol_{A_0}(E\times \{0\}+\{0\}\times E + \tilde\Delta)
 $$
 where
 $$
 \tilde\Delta := \{(z,-z): z \in E\}
 $$
 is the \emph{antidiagonal}. $\sL_0$ defines an irreducible polarization of type $(1,3)$ on $A_0$.  We choose canonical coordinates $(y_1,y_2,y_3)$  in $|\sL_0|$ so that $E$ is a cubic $E_\lambda$ of the Hesse pencil. With this choice the action of the   group  $\hat G$ on $\PP^2$ as studied in \S \ref{Sheishess} is the same as the action of $\hat G$ on the linear system $|\sL_0|$. Then:

 \begin{enumerate}
     \item The branch curve $\sB\subset \PP^{2\vee}$ is equal to $3 E^{\vee}$, where $E^{\vee}$ is the sextic dual of $E$.

     \item The ramification divisor   of $\varphi_{\sL_0}:A_0 \ra \PP^{2\vee}$ is
     $$
     R_0 = \Delta + \De_{-2}+\De^t_{-2}
     $$
     where $\Delta= \{(z,z): z\in E\}\subset A_0$ is the diagonal, $\De_{-2}: =\{(z,-2z):z\in E\}\subset A_0$ is the graph of $(-2)_{|E}$, and $\De^t_{-2}: = \{(-2z,z):z\in E\}$ its transpose.

     \item  $$
     \Delta \cap \De_{-2}=\Delta \cap \De_{-2}^t=\De_{-2}\cap \De_{-2}^t=\{(z,z): 3z=0\}
     $$
     consists of 9 ordinary triple points for $R_0$. In particular 
     $p_a(R_0)=28$.
     
     \item
       $\varphi_{\sL_0}:A_0 \ra \PP^{2\vee}$ is a Galois covering with Galois group $\mathfrak S_3$, and the divisors $ \Delta , \De_{-2}, \De^t_{-2}$ are the respective sets
 of fixpoints of the three different elements of order two in $\mathfrak S_3$.
 
 The 9 triple points of $R_0$  are the points which are fixed by the entire group, and they map to the 9 cusps of $\sB$. 
 
 The morphism $\varphi_{\sL_0}$ factors through
 a $\ZZ/3$-Galois covering $ A_0 \ra \sK$, where $\sK$ is the singular K3 surface with
 9 points of type $A_2$ which is the double covering of the plane branched on $\sB$.

     \item\label{duality} $\varphi_{\sL_0|\Delta}:  \Delta \ra E^\vee$   coincides with the duality map 
     $$
     \Delta=E \ra E^\vee,\quad z \mapsto \text{tangent line at $z$}.
     $$
     
     \item
     The  pencil of lines through a cusp of $\sB$ yields a pencil in $|D|$ of curves
     which have the same singular point, at the inverse image of the cusp, and  with  general singularity  a node. These pencils yield 9 lines $\ell_i$ in the incidence
     curve $\Ga$ and also 9 lines $L_i$ contained in the discriminant curve $W_0$.

     \item 
     
     The incidence curve is set theoretically
     $$
     \Gamma_0 = R_0\cup \ell_1\cup\cdots\cup\ell_9
     $$
 where the $\ell_i$'s are pairwise disjoint lines. 
 
 Over one of the above 9 points $P_i $ in $A_0$ (3-torsion points of the diagonal $\De$),
 there pass 4 smooth components of $\Ga_0$.
 
 To analyse the local geometry, we choose local coordinates $(z_1, z_2)$ at $P_i$
 such that the $\mathfrak S_3$ quotient can be written via the geometry of a degree 3 polynomial, getting coordinates $p,q$ at the cusp such that 
 $$  t ^3 +  p t - q = (t-z_1) (t-z_2) ( t + z_1 + z_2), \Rightarrow \sB = \{ 4 p^3 + 27 q^2=0\}.$$
 
 Then the curves of $|D|$ are locally given by
 $$   D (a,b,c) = \{  a ( z_1 z_2 - (  z_1 + z_2 )^2)  - b z_1 z_2 ( z_1 + z_2) +c =0\}.$$

Now, $z$ is singular for $ D(a,b,c) $ means that 
$$
( 2 z_1 + z_2)  ( b z_2 + a) =0  , \ \ 
 ( 2 z_2 + z_1) ( b z_1 + a) = 0.$$
 
 The solution $z_1=z_2=0$ yields the curve $\ell_i$,
 and we can otherwise set $b=1$ (the cusp tangent is for $a=0$) 
 and see that $\Ga$ is defined as the complete intersection in $\CC^3$
 (with coordinates $(z_1, z_2, a)$)
  of two pair of planes.

The union of the other three lines maps (not isomorphically) in $A_0$ to 
the union of three lines:

$$ 	(z_2 + 2 z_1 ) (z_1 - z_2 ) (z_1 + 2 z_2)=0.$$

The discriminant curve is locally given by  
$$  b=1, c = 5  a^3 ,$$

and we have indeed a flex point of the reduced discriminant.

 Note that,  since $\Ga_0$ is Gorenstein
 (being a local complete intersection) and reduced, we could have told a priori
 that, at the point $(z_1,z_2, a, c)= (0,0,0,0)$,
 the Zariski tangent dimension    can be only  2 or 3 (corresponding to the fact that a Gorenstein configuration of 4 points in a projective space consists either of collinear points or of a projective base in the plane), and  it could not
 have been 2 else the local contribution to the number $\de$ of double points would have been 6; it then follows  that the tangent dimension is 3  and  that locally $\Ga$ is a complete
 intersection as described by the above equations.

       \item The discriminant curve in $\PP^2$ is
     $$
     W_0=3E+L_1+ \cdots + L_9,
     $$
where $L_1, \dots, L_9$ are the flex tangents of $E$.
\item\label{2duality} $\pi_2: \Gamma_0\ra W_0$ maps  $\ell_1\cup\cdots\cup\ell_9$ onto $L_1\cup\cdots\cup L_9$ and  $R_0$ onto $E$.
Its restriction to $\Delta$ is the composition $\Delta \ra E^{\vee}  \ra E$   where the first map is \eqref{duality} and the second is again the duality map. A similar behaviour occurs on the components $\De_{-2}$ and $\De_{-2}^t$ because of the action of the Galois group $\mathfrak S_3$. The 9 triple points of $R_0$ are mapped to the flexes, and the map $\pi_{2|R_0}$ factors as
$$
R_0 \ra E^{\vee} \ra E
$$
where the first map is $\varphi_{\sL_0}$ and the second map is duality.

     \item $\overline V_{\sL_0,1}= L_1 \cup \cdots \cup L_9$. In particular $\overline V_{\sL_0,1}\ne W_0$.

 \end{enumerate}


 \subsection{The degeneration of the singular points.}

 We now consider a 1-parameter family of polarized surfaces of type $(1,3)$, parametrized by a pointed curve $(U,0)$ with general fibre $(A,\sL)$ and special fibre 
 $(A_0,\sL_0)$ as described in subsection \ref{SSspecial}. Here we assume that the elliptic curve $E$ is not equianarmonic. There is an induced family of incidence curves and of discriminant curves and maps between them:
 $$
\xymatrix{
\widetilde\Gamma \ar[dr]\ar[rr]^-{\Pi_2}&&\widetilde W \ar[dl]\\
&U}
$$
 For a general $t \in U$ the morphism $\pi:=\Pi_2(t)$ is birational  and $\Gamma=\widetilde \Gamma(t)$ is nonsingular. The cusps of $W$ are precisely the images of the points of $\Gamma$  where $d\pi$ degenerates.

 When $(A,\sL)$ tends to $(A_0,\sL_0)$ every such point must tend to a point where the differential of the map 
 $$
 \pi_2:\Gamma_0 = R_0 \cup \ell_1\cup\cdots\cup\ell_9 \ra W_0
 $$
  degenerates or $\Gamma_0$ is singular. Certainly $d\pi_2$ does not degenerate at the nonsingular points of $W_0$ along the components $\ell_1, \dots, \ell_9$. The restriction $\pi_{2|R_0}$ is the composition $R_0 \ra E^{\vee}  \ra E$ (compare \eqref{2duality}). The differential of this map degenerates precisely at the 9 points   mapped to the cusps of $E^{\vee}$, and these are mapped to the inflection points. Therefore \emph{the 72 cusps of $W$ tend to the 9 inflection points  of $E\subset W_0$}.
  
  Another way to establish the claim about the limits of the cusps is to  observe that if a smooth point $P'$ of $E_{\la}$ 
were a limit of cusps, then it would also be the limit of a smooth point, and we would have two branches
whose local monodromy are respectively a couple of non commuting transpositions, and a simple transposition.
The corresponding permutations would  therefore be contained in a subgroup isomorphic to $\mathfrak S_5$.
But the local monodromy in the limit equals to the product of three pairwise disjoint 
transpositions, which cannot therefore be contained in a subgroup  isomorphic to $\mathfrak S_5$.

   Since the elliptic curve $E$ is not equianarmonic the inflectional tangents $L_1, \dots, L_9$ meet in 36 points which are nodes of $W_0$ (compare subsection \ref{SSintersections}). We claim that when we deform $(A_0,\sL_0)$ these 36 nodes cannot be smoothed and therefore must deform into   36 nodes of $W$. Thus  \emph{the nodes of $W$ tend to the 36 pairwise intersections of the inflectional tangents of $E$.}
   
   This shows in particular that $W$ has precisely 36 distinct nodes, and therefore it has no tacnodes.

   \medskip
   
 \textbf{Proof of the Claim.}   Each node $P_{ij}$ of $W_0$ belongs to the intersection
 of two flex tangents $L_i, L_j$, which in turn are the respective isomorphic images of the smooth rational curves $\ell_i, \ell_j$,
 which are disjoint. Let  $P_i(j) \in \ell_i, P_j (i)  \in \ell_j$   the two inverse images of $P_{ij}$.
 
 Deforming $A_0$, also $R_0$ deforms to $R_t$, and the two branches at 
 $P_i(j) $, respectively at $ P_j (i) $ deform to two disjoint branches.
 
 Their respective images in  $\PP^2$ necessarily intersect,  and we have an equisingular deformation of the node.

     
     \bigskip
     
     \section{A  new component of the moduli space of surfaces of general type, consisting of surfaces with $p_g=q=2$, $K_S^2=6$, $d=3$.}

     Our main aim is to establish, in the case of the curve $ C : = \{ f(y) : = \sum_i y_i^2 y_{i+1} =0\}$, 
      that the general surface $S'$ is smooth,
      where 
       $$ S' : = \{ (y,z) \in \PP( V) \times A | \sum_j y_j x_j(z) = 0 , \sum_i y_i^2 y_{i+1} =0\}, \ S' \subset A \times C.$$

     For the reader's convenience, we reproduce from \cite{al-cat}
     the  discussion of  the singularities of $S'$, where we have shown that $C$ is smooth.
     
      We first observe  that     $C$ does not contain any of the flexpoints. Since $C$ is $G$-invariant,
      and the flex-points are a $G$-orbit, it suffices to verify that $(1,-1,0) \notin C$, this is clear since 
      for $ f =  \sum_i y_i^2 y_{i+1} $, $ f(1,-1,0)= -1$. 
      
       Moreover, $C$ intersects transversally any smooth curve  $E_{\la}$ of the Hesse pencil:
      this is a direct consequence of Bertini's theorem for general $\la$, but follows also for each
      $\la$  because $ C \cap E_{\la}$ is $G$-invariant, and $G$ acts freely on $E_{\la}$,
      hence $ C \cap E_{\la}$ consists of nine distinct points.
      
      Moreover, $C$ does not intersect the 36 nodes of the curve $W_0$, for general choice of $\la$,
      since these points sweep out the nine lines $ L^*_i : = \{ y_i = \e^j y_{i+1}\}$, hence 
      there is an injective correspondence $\la \mapsto L_i (\la) \cap L_j (\la)$.
      
      Finally, for general $\la$, $C$ intersects transversally each of the nine lines $L_1, \dots, L_9$
      which are the flex tangents of $E_{\la}$. Since the nine lines are the orbit of $L_1$ under the
      action of $G$, it suffices to show our assertion for $L_1$ and for some value of $\la$.
      
      If we intersect $C$ with the tangent at $(1, -1, 0)$ to the Fermat cubic $E_0$, 
       then we get 
      $$ y_1 + y_2 = \sum_i y_i^2 y_{i+1} =0 \Leftrightarrow y_2 = - y_1 , - y_1^3 + y_1^2 y_3 + y_3^2 y_1=0,$$
      that is, 
      $$y_2 = - y_1 , \ {\rm and } \ y_1 ( y_3^2 + y_1 y_3 - y_1^2) =0,$$ 
      and these are three distinct points.

      \subsection{The general surface $S'$ is smooth} Passing to the smoothness of $S'$, 
 	$(y,z)$ is  a singular point of $S'$ if and only if  $z$ is  a singular point of the curve $D_y:=\{z | \sum_j y_j x_j(z)=0\}$ and the rows of the matrix
 	
 	\begin{equation}
 		\begin{pmatrix}
 			y_3^2 + 2y_1y_2 & y_1^2 + 2y_2y_3 &  y_2^2 + 2y_1y_3  \\
 			x_1 & x_2& x_3
 		\end{pmatrix}
 	\end{equation}
	are  proportional. This means  that
	
	\begin{equation}
		x: = (x_1, x_2,x_3)= \nabla f (y), \ y: = (y_1, y_2, y_3)
	\end{equation}
	
	and 	we view $x$ as a point of $(\PP^2)^{\vee}= : \PP'$, while $ y \in \PP : = \PP^2$.
	
Geometrically, this means that $ x \in C^{\vee}$, and $x$ represents a tangent line to $C$ at $y$,
hence $y$ represents a line $\Lam_y$ tangent to $C^{\vee}$ at $x$.

Moreover, since $z$ is a singular point of $D_y$, which is the inverse image under $\varphi_D$
of the line $\Lam_y$ corresponding to $y$, we require  that the	 line $\Lam_y$ is tangent at $x$
to the branch curve $\sB$ of $\varphi_D$. Hence, that $\sB$ and $C^{\vee}$ are tangent.

Hence, $S'$ is smooth if $\sB$ and $C^{\vee}$ intersect transversally.

By duality, this is equivalent to the assertion that $C$ intersects transversally the discriminant curve $W$.

We prove  the following
\begin{prop}\label{transversality}
Let $\sB$ be the branch curve of $\varphi_D : A \ra \PP^2$, where $D$ is a polarization of type $(1,3)$
and the pair $(A, D)$ is general.

Then, $C$ being the plane curve $ C : = \{  \sum_i y_i^2 y_{i+1} =0\}$, 
$\sB$   intersects transversally  the dual sextic curve $C^{\vee}$ and $C$ intersects transversally
the discriminant curve $W$.
\end{prop}
\begin{proof}

Let $W_0$ be the discriminant curve in the degenerate case where $A = E\times E$ and $E =: E_{\la}$ is 
not isomorphic to the Fermat cubic. 

Then $W_0$ consists of $E_{\la}$ counted three times, plus the nine flex tangents $L_1, \dots, L_9$.

If $\la$ is general, then  the cubic curve $C$ is transversal to the reduced curve $(W_0)_{red} = E_{\la}
\cup L_1 \cup \dots \cup L_9$.

If $W_t$ is a small deformation of $W_0$, then $C$ does not contain any singular point of $W_t$,
since the 36 nodes of $W_t$ tend to the 36 nodes of $W_0$, while  the cusps of $W_t$ tend to
the flexpoints of $E_{\la}$: and $C$ does not pass through these nodes of $W_0$ and these flexpoints.

Since $C$ intersects $E_{\la}$ transversally at 9 smooth points $P'_i$ , deforming $W_0$
to $W_t$ we obtain three smooth branches in the neighbourhood of $P'_i$ which remain transversal
to $C$ and do not intersect each other.

\end{proof}

\bigskip

\bigskip

\bigskip

\bigskip

\bigskip

{\bf Acknowledgement:} Thanks to Gianfranco Casnati for pointing out to the first author the work of the second author together with Herbert Lange.

\end{document}